\documentclass[10pt]{amsart}
\usepackage{amssymb,mathrsfs,skak}

\def\bbb {\mathbf{b}}
\def\bC {\mathbf{C}}

\def\bN {\mathbf{N}}

\def\bR {\mathbf{R}}
\def\bS {\mathbf{S}}

\def\bZ {\mathbf{Z}}

\def\cL {\mathcal{L}}
\def\cM {\mathcal{M}}
\def\cN {\mathcal{N}}

\def\cS {\mathcal{S}}

\def\scrL{\mathscr{L}}

\def\a {{\alpha}}

\def\g {{\gamma}}

\def\de {{\delta}}
\def\eps {{\epsilon}}
\def\th {{\theta}}
\def\ka {{\kappa}}
\def\l {{\lambda}}
\def\L {{\Lambda}}
\def\si {{\sigma}}

\def\vphi {{\varphi}}

\def\d {{\partial}}
\def\grad {{\nabla}}
\def\Dlt {{\Delta}}

\def\rstr {{\big |}}
\def\indc {{\bf 1}}

\newcommand{\Diag}{\operatorname{diag}}

\newcommand{\Dim}{\operatorname{dim}}

\newcommand{\Sign}{\operatorname{sign}}

\newcommand{\Supp}{\operatorname{supp}}
\newcommand{\Det}{\operatorname{det}}
\newcommand{\Tr}{\operatorname{trace}}
\newcommand{\Id}{\operatorname{Id}}

\newcommand{\Ker}{\operatorname{Ker}}

\newcommand{\ba}{\begin{aligned}}
\newcommand{\ea}{\end{aligned}}

\newcommand{\be}{\begin{equation}}
\newcommand{\ee}{\end{equation}}

\newcommand{\lb}{\label}

\newtheorem{Thm}{Theorem}[section]

\newtheorem{Prop}[Thm]{Proposition}

\newtheorem{Lem}[Thm]{Lemma}

\begin{document}

\title[Classical Limit of Schr\"odinger]{On the Classical Limit \\ of the Schr\"odinger Equation}

\author[C. Bardos]{Claude Bardos}
\address[C.B.]{Universit\'e Paris-Diderot, Laboratoire J.-L. Lions, BP187, 4 place Jussieu, 75252 Paris Cedex 05 France}
\email{claude.bardos@gmail.com}

\author[F. Golse]{Fran\c cois Golse}
\address[F.G.]{Ecole Polytechnique, Centre de Math\'ematiques L. Schwartz, 91128 Palaiseau Cedex, France}
\email{francois.golse@math.polytechnique.fr}

\author[P. Markowich]{Peter Markowich}
\address[P.M.]{King Abdullah University of Science and Technology (KAUST), MCSE Division, Thuwal 23955-6900, Saudi Arabia}
\email{Peter.Markowich@kaust.edu.sa}

\author[T. Paul]{Thierry Paul}
\address[T.P.]{Ecole Polytechnique, Centre de Math\'ematiques L. Schwartz, 91128 Palaiseau Cedex, France}
\email{thierry.paul@math.polytechnique.fr}

\begin{abstract}
This paper provides an elementary proof of the classical limit of the Schr\"odinger equation with WKB type initial data and over arbitrary long finite time intervals. We use only the stationary phase method and the Laptev-Sigal 
simple and elegant construction of a parametrix for Schr\"odinger type equations [A. Laptev, I. Sigal, Review of Math. Phys. \textbf{12} (2000), 749--766]. We also explain in detail how the phase shifts across caustics obtained 
when using the Laptev-Sigal parametrix are related to the Maslov index.
\end{abstract}

\keywords{Schr\"odinger equation, Classical limit, WKB expansion, Caustic, Fourier integral operators, Lagrangian manifold, Maslov index.}

\subjclass{Primary: 35Q41, 81Q20; Secondary: 35S30, 53D12.}

\maketitle


\section{The classical scaling}


Consider the evolution Schr\"odinger equation
$$
i\hbar\d_t\psi=-\tfrac{\hbar^2}{2m}\Dlt_x\psi+V(x)\psi
$$
for the wave function $\psi$ of a point particle of mass $m$ subject to the action of an external potential $V\equiv V(x)\in\bR$.
 
Choosing ``appropriate'' units of time $T$ and length $L$, we recast the Schr\"odinger equation in terms of dimensionless variables $\hat t:=t/T$ and $\hat x:=x/L$. We define a rescaled wave function $\hat\psi$ and a rescaled, 
dimensionless potential $\hat V$ by the formulas 
$$
\hat\psi(\hat t,\hat x):=\psi(t,x)\quad\hbox{ and }\hat V(\hat x):=\frac{T^2}{mL^2}V(x)\,.
$$
In these dimensionless variables, the Schr\"odinger equation takes the form
$$
i\d_{\hat t}\hat\psi=-\frac{\hbar T}{2mL^2}\Dlt_{\hat x}\hat\psi+\frac{mL^2}{\hbar T}\hat V(\hat x)\hat\psi\,.
$$
The dimensionless number $2\pi\hbar T/mL^2$ is the ratio of the Planck constant to $mL^2/T$, that is (twice) the action of a classical particle of mass $m$ moving at speed $L/T$ on a distance $L$. If the scales of time $T$ and length 
$L$ have been chosen conveniently, $L/T$ is the typical order of magnitude of the particle speed, and $L$ is the typical length scale on which the particle motion is observed. The classical limit of quantum mechanics is defined by the 
scaling assumption $2\pi\hbar\ll mL^2/T$ --- i.e. the typical action of the particle considered is large compared to $\hbar$. Equivalently, $mL/T$ is the order of magnitude of the particle momentum, so that $2\pi\hbar T/mL$ is its de Broglie 
wavelength; the scaling assumption $2\pi\hbar T/mL\ll L$ means that the de Broglie wavelength of the particle under consideration is small compared to the observation length scale $L$.

Introducing the small, dimensionless parameter $\eps=\hbar T/mL^2$ and dropping hats in the dimensionless variables as well as on the rescaled wave function and dimensionless potential, we arrive at the following formulation for the 
Schr\"odinger equation in dimensionless variables
\be\lb{ScalSchro}
i\eps\d_t\psi=-\tfrac12\eps^2\Dlt_x\psi+V(x)\psi\,.
\ee

The WKB ansatz postulates that, at time $t=0$, the wave function is of the form
$$
\psi(t,x)=a^{in}(x)e^{iS^{in}(x)/\hbar}\,,\quad x\in\bR^N\,.
$$
Consistently with the scaling argument above, we set 
$$
\hat a^{in}(\hat x):=a^{in}(x)\hbox{ and }\hat S^{in}(\hat x):=TS^{in}(x)/mL^2
$$
 --- since $S^{in}$ has the dimension of an action --- so that
\be\lb{ScalIn}
\hat\psi(0,\hat x)=\hat a^{in}(\hat x)e^{i\hat S^{in}(\hat x)/\eps}\,.
\ee

Dropping hats in the initial data as well as in the Schr\"odinger equation, one arrives at the following Cauchy problem for the Schr\"odinger equation in dimensionless variables:
\be\lb{ClassSchrod}
\left\{
\ba
{}&i\eps\d_t\psi_\eps=-\tfrac12\eps^2\Dlt_x\psi_\eps+V(x)\psi_\eps\,,\quad x\in\bR^N\,,\,\,t\in\bR\,,
\\
&\psi_\eps(0,x)=a^{in}(x)e^{iS^{in}(x)/\eps}\,.
\ea
\right.
\ee

The problem of the classical limit of the Schr\"odinger equation is to describe the wave function $\psi_\eps$ in the limit as $\eps\to 0^+$.

\smallskip
The main result obtained in this paper (Theorem \ref{T-WKBV}) is of course not new --- it is stated without proof as Theorem 5.1 in \cite{ArnoldMaslov}; see also Example 6.1 described on pp. 141-143 in \cite{Masl}, especially formula 
(1.21) there. 

Yet our purpose in the present work is to provide a short and self-contained proof of this result, based on Laptev-Sigal lucid construction of a parametrix for the Schr\"odinger equation (\ref{ScalSchro}) in \cite{LapSig}. In this way, we 
avoid using either Maslov's formalism \cite{Masl,MaslFed} of the canonical operator, or the semiclassical analogue of H\"ormander's global theory \cite{Hor4} of Lagrangian distributions or Fourier integral operators. The Laptev-Sigal
construction, which is is remarkably short (pp. 753--759 in \cite{LapSig}), uses instead Fourier integrals with a complex phase whose imaginary part is quadratic. 


\section{The classical dynamics}


Before discussing the classical limit of the Schr\"odinger equation, we need to recall a few preliminary results on either the quantum dynamics defined by (\ref{ClassSchrod}) or the classical dynamics in phase space associated 
to the Schr\"odinger equation (\ref{ClassSchrod}).

\smallskip
Assume that $V\in C^\infty(\bR^N)$ satisfies
\be\lb{Cinftyb}
\d^\a V\in L^\infty(\bR^N)\quad\hbox{ for each multi-index }\a\in\bN^N\hbox{ such that }|\a|>0
\ee
and 
\be\lb{SublinV}
\frac{V(x)}{|x|}\to 0\quad\hbox{ as }|x|\to+\infty\,.
\ee
Then the Hamiltonian 
$$
H(x,\xi):=\tfrac12|\xi|^2+V(x)
$$ 
generates a global flow 
$$
\bR^N\times\bR^N\ni(x,\xi)\mapsto\Phi_t(x,\xi):=(X_t(x,\xi),\Xi_t(x,\xi))\in\bR^N\times\bR^N
$$
that is of class $C^1$. This Hamiltonian flow is the classical dynamics corresponding to the quantum dynamics defined by the Schr\"odinger equation (\ref{ClassSchrod}). 

This classical dynamics satisfies the following properties: for each $\eta>0$, there exists $C_\eta>0$ such that
$$
\sup_{|t|\le T}|X_t(x,\xi)-x|\le C_\eta(1+|\xi|)+\eta|x|
$$
for each $x,\xi\in\bR^N$, and
$$
|D\Phi_t(x,\xi)-\Id_{\bR^N\times\bR^N}|\le e^{\ka|t|}-1
$$
for all $t\in\bR$. (See Lemma 4.1 in \cite{BGMP}.)

Assume further that
\be\lb{SA1}
\sup_{x\in\bR^N}\int_{\bR^N}\Gamma_\eta(x-y)V^-(y)dy\to 0\hbox{ as }\eta\to 0\quad\hbox{ if }N\ge 2
\ee
with 
$$
\Gamma_\eta(z)=\left\{\ba{}&\indc_{[0,\eta]}(|z|)|z|^{2-N}&&\hbox{ if }N\ge 3\,,\\&\indc_{[0,\eta]}(|z|)\ln(1/|z|)&&\hbox{ if }N=2\,,\ea\right.
$$
while
\be\lb{SA2}
\sup_{x\in\bR^N}\int_{x-1}^{x+1}V^-(y)dy<\infty\quad\hbox{ if }N=1\,.
\ee
Under assumptions (\ref{SA1})-(\ref{SA2}), the operator $-\tfrac12\eps^2\Dlt_x+V$ has a self-adjoint extension on $L^2(\bR^N)$ that is bounded from below. (See \cite{LionsPaul} on p. 567.)

Next we discuss the properties of the initial phase function $S^{in}$. It will be convenient to assume that $S^{in}$ is of class $C^2$ at least on $\bR^N$ and satisfies the following growth condition at infinity:
\be\lb{Sublin}
\frac{|\grad S^{in}(y)|}{|y|}\to 0\qquad\hbox{ as }|y|\to 0\,.
\ee
Consider the map
\be\lb{DefF}
F_t:\,\bR^N\ni y\mapsto F_t(y)=X_t(y,\grad S^{in}(y))\in\bR^N\,,
\ee
together with the absolute value of its Jacobian determinant
\be\lb{DefJ}
J_t(y)=|\Det(DF_t(y))|\,.
\ee

We also introduce the set
\be\lb{Caustic}
\ba
{}&C\,:=\{(t,x)\in\bR\times\bR^N\,|\,F_t^{-1}(\{x\})\cap J_t^{-1}(\{0\})\not=\varnothing\}\,,
\\
&C_t:=\{x\in\bR^N\,|\,(t,x)\in C\}\,,\quad t\in\bR\,.
\ea
\ee
For lack of a better terminology and by analogy with geometric optics, $C$ will be referred to as the ``caustic'' set, and $C_t$ as the ``caustic fiber''.

\begin{Prop}\lb{P-Smooth}
Assume that (\ref{Sublin}) holds, and that
$$
|\grad^2S^{in}(y)|=O(1)\quad\hbox{ as }|y|\to+\infty\,.
$$

\smallskip
\noindent
(a) For each $t\in\bR$, the map $F_t$ is proper and onto. More precisely, for each $T>0$
\be\lb{MajF}
\sup_{|t|\le T}\frac{|F_t(y)-y|}{|y|}\to 0\quad\hbox{ as }|y|\to\infty\,.
\ee

\noindent
(b) The caustic set $C$ is closed in $\bR\times\bR^N$, and $\scrL^N(C_t)=0$ for each $t\in\bR$.

\noindent
(c) For each $(t,x)\in\bR\times\bR^N\setminus C$, the set $F_t^{-1}(\{x\})$ is finite, henceforth denoted by
$$
\{y_j(t,x)\,,\,\,j=1,\ldots,\cN(t,x)\}\,.
$$
The integer $\cN$ is a constant function of $(t,x)$ in each connected component of $\bR\times\bR^N\setminus C$ and, for each $j\ge 1$, the map $y_j$ is of class $C^1$ on each connected component of $\bR\times\bR^N\setminus C$ 
where $\cN\ge j$.

\noindent
(d) There exists $a<0<b$ such that $C\cap((a,b)\times\bR^N)=\varnothing$ and $\cN=1$ on $(a,b)\times\bR^N$.

\noindent
(e) For each $(t,x)\in\bR\times\bR^N\setminus C$, the integer $\cN(t,x)$ is odd.
\end{Prop}

(See Theorems 2.3 and 2.5 in \cite{BGMP} for a proof of these results.)

\smallskip
One might be worried that the sublinearity assumption (\ref{SublinV}) obviously excludes the harmonic oscillator, i.e. the case where $V(x):=\tfrac12|x|^2$. In view of the importance of the harmonic oscillator in quantum mechanics, the 
setting presented in this section might seem unfelicitous. In fact, the main reason for assuming (\ref{SublinV}) is that the set $F_t(\{x\})$, which consists of isolated points only, is necessarily finite because $F_t$ is proper so that $F_t(\{x\})$
is a compact set with only isolated points, i.e. a finite set. For that reason, our main results Theorems \ref{T-WKBV=0} and \ref{T-WKBV} hold without additional assumption on the initial amplitude $a^{in}$. If (\ref{SublinV}) is not satisfied, 
one cannot in general conclude that $F_t(\{x\})$ is finite. But one can choose to consider only the case of initial amplitudes $a^{in}$ with compact support, so that $F_t(\{x\})\cap\Supp(a^{in})$ is necessarily finite. With this slight modification, 
most of the results of the present paper also apply to the harmonic oscillator.


\section{The free case $V\equiv 0$.}


For the sake of clarity, we first discuss in detail the case $V\equiv 0$. This section is the exact analogue for the Schr\"odinger equation of chapter XII.1 in \cite{Hor2}, which presents geometric optics as an asymptotic theory for solutions 
of the wave equation with constant coefficients in the high frequency limit.

Assume that $a^{in}\in C_c(\bR^N)$ while $S^{in}\in C^2(\bR^N)$ and satisfies (\ref{Sublin}). The solution $\psi_\eps$ of (\ref{ClassSchrod}) is given by the explicit formula
\be\lb{Sol}
\psi_\eps(t,x)=\left(\frac{1}{\sqrt{2\pi i\eps t}}\right)^N\int_{\bR^N}e^{i\phi(t,x,y)/\eps}a^{in}(y)dy\,,\quad x\in\bR^N\,,\,\,t\not=0\,,
\ee
where $\sqrt{}$ designates the holomorphic extension of the square root to $\bC\setminus\bR_-$ and
\be\lb{Phase}
\phi(t,x,y):=\frac{|x-y|^2}{2t}+S^{in}(y)\,,\qquad x,y\in\bR^N\,,\,\,t\not=0\,.
\ee

\smallskip
Applying the stationary phase method (see Theorems 7.7.5 and 7.7.6 in \cite{Hor1}) to the explicit formula (\ref{Sol}) leads to the following classical result on the asymptotic behavior of $\psi_\eps$ as $\eps\to 0^+$.

\begin{Thm}\lb{T-WKBV=0}
Let $S^{in}\in C^{m+1}(\bR^N)$ satisfy (\ref{Sublin}) and $a^{in}\in C^m_c(\bR^N)$ with $m\ge\tfrac32N+3$. For all $(t,x)\in\bR\times\bR^N\setminus C$ and $\eps>0$, the solution $\psi_\eps$ of the Cauchy problem (\ref{ClassSchrod}) 
satisfies
\be\lb{WKBLim}
\psi_\eps(t,x)\!-\!\sum_{j=1}^{\cN(t,x)}\!\frac{i^{-\cM_j(t,x)}}{|\Det(I\!+\!t\grad^2S^{in}(y_j(t,x)))|^{1/2}}a^{in}(y_j(t,x))e^{i\phi(t,x,y_j(t,x))/\eps}=O(\eps)
\ee
uniformly as $(t,x)$ runs through compact subsets of $\bR\times\bR^N\setminus C$, where $\cM_j(t,x)$ is the number of negative eigenvalues of the matrix $I+t\grad^2S^{in}(y_j(t,x))$.
\end{Thm}

\begin{proof}
Pick $(t,x)\in\Supp(a^{in})$; critical points $y$ of the phase $\phi$ are obtained by solving for $y$ the equation
$$
\grad_y\phi(t,x,y)=\frac{y-x}{t}+\grad S^{in}(y)=0\,,
$$
i.e.
$$
x=y+t\grad S^{in}(y):=F_t(y)\,.
$$

If $(t,x)\notin C$, the equation above has finitely many solutions denoted 
$$
y_1(t,x),\dots,y_{\cN(t,x)}
$$
(see Proposition \ref{P-Smooth} above). For $j=0,\ldots,\cN(t,x)$, let $\chi_j\in C^\infty_c(\bR^N)$ satisfy
$$
\chi_j\equiv 1\hbox{ near }y_j(t,x)\,,\quad j=1,\ldots,\cN(t,x)\,,\quad F_t^{-1}(\{x\})\cap\Supp(\chi_0)=\varnothing
$$
and
$$
\chi_j\ge 0\hbox{ for all }j=0,\ldots,\cN(t,x)\,,\qquad\sum_{j=0}^{\cN(t,x)}\chi_j\equiv 1\hbox{ near }\Supp(a^{in})\,.
$$
Thus
\be\lb{Sumj0}
\psi_\eps(t,x)=\left(\sqrt{2\pi i\eps t}\right)^{-N}\sum_{0\le j\le\cN(t,x)}\int_{\bR^N}e^{i\phi(t,x,y)/\eps}a^{in}(y)\chi_j(y)dy\,.
\ee

In the term corresponding with $j=0$, the phase $\phi(t,x,\cdot)$ has no critical point on $\Supp(\chi_0 a^{in})$. By Theorem 7.7.1 in \cite{Hor1}, one has
\be\lb{NonStatPh0}
\ba
\left|\int_{\bR^N}e^{i\phi(t,x,y)/\eps}a^{in}(y)\chi_0(y)dy\right|&
\\
\le(\eps|t|)^kC_0(t\|\grad^{k+1}S^{in}\|_{L^\infty(B(0,R))})\sum_{|\a|\le k}\sup_{y\in\bR^N}\frac{|\d^\a(a^{in}\chi_0)(y)|}{|y+t\grad S^{in}(y)-x|^{|\a|-2k}}&\,,
\ea
\ee
where $R>0$ is chosen so that $\Supp(a^{in})\subset B(0,R)$ and the function $C_0$ is bounded on bounded sets of $\bR_+$. Thus
$$
\left|\left(\sqrt{2\pi i\eps t}\right)^{-N}\int_{\bR^N}e^{i\phi(t,x,y)/\eps}a^{in}(y)\chi_0(y)dy\right|=O(\eps|t|)
$$
provided that one can choose $k\ge\tfrac12N+1$ in the previous inequality. Applying Theorem 7.7.1 of \cite{Hor1} for this estimate requires $m\ge k\ge\tfrac12N+1$.

For $j=1,\ldots,\cN(t,x)$, one has
$$
\ba
\left|\int_{\bR^N}e^{i\phi(t,x,y)/\eps}a^{in}(y)\chi_j(y)dy-(2\pi\eps)^{N/2}\frac{(a^{in}\chi_j)(y_j(t,x))e^{i\phi(t,x,y_j(t,x))/\eps}}{\sqrt{(-i)^N\Det\grad^2_y\phi(t,x,y_j(t,x))}}\right|
\\
\le(\eps|t|)^kC_1\left(\sup_{\chi_j(y)>0}\frac{t|y-y_j(t,x)|}{|y+t\grad S^{in}(y)-x|}+t\|\grad^{3k+1}S^{in}\|_{L^\infty(B(0,R))}\right)
\\
\times\sum_{|\a|\le k}\sup_{y\in\bR^N}|\d^\a(a^{in}\chi_0)(y)|
\\
+(2\pi\eps)^{N/2}\frac{|(a^{in}\chi_j)(y_j(t,x)|}{|\Det\grad^2_y\phi(t,x,y_j(t,x))|^{1/2}}\sum_{1\le j<k}(\eps|t|)^jL_ja^{in}
\ea
$$
by Theorem 7.7.5 in \cite{Hor1}, where the function $C_1$ is bounded on bounded sets of $\bR_+$. In the inequality above, we have denoted
$$
L_ja^{in}:=\sum_{\nu-\mu=j\atop2\nu\ge3\mu}\frac1{2^\nu\mu!\nu!}|((I+t\grad^2S^{in}(y_j(t,x))^{-1}:\grad_y^{\otimes 2})^\nu(a^{in}T^2_{y_j}\phi(t,x,\cdot))(y_j(t,x))|\,,
$$
and
$$
T^2_{y}f(x):=f(x)-f(y)-f'(y)\cdot(x-y)-\tfrac12f''(y)\cdot(x-y)^2\,.
$$
Thus, Theorem 7.7.5 of \cite{Hor1} shows that
\be\lb{StatPh0}
\ba
\left|\left(\sqrt{2\pi i\eps t}\right)^{-N}\int_{\bR^N}e^{i\phi(t,x,y)/\eps}a^{in}(y)\chi_j(y)dy-\frac{a^{in}(y_j(t,x))e^{i\phi(t,x,y_j(t,x))/\eps-iN\pi/4}}{\sqrt{(-i)^N\Det(I+t\grad^2S^{in}(y_j(t,x))}}\right|
\\
=O(\eps|t|)
\ea
\ee
provided that one can choose $k\ge\tfrac12N+1$ in the previous inequality. This requires $m\ge 3k\ge\tfrac32N+3$.

It remains to identify $\cM_j(t,x)$. The matrix $\grad^2S^{in}(y)$ is symmetric with real entries, and therefore reducible to diagonal form with real eigenvalues
$$
\l_1(y)\ge\l_2(y)\ge\ldots\ge\l_N(y)
$$
counted with multiplicities. Since $(t,x)\notin C$, one has $1+t\l_k(y_j(t,x))\not=0$ for each $k=1,\ldots,N$ and each $j=1,\ldots,\cN(t,x)$. Set 
$$
\cM_j(t,x):=\#\{1\le k\le N\hbox{ s.t. }1+t\l_k(y_j(t,x))<0\}\,.
$$
Thus
$$
\Det(I+t\grad^2S^{in}(y_j(t,x))=(-1)^{\cM_j(t,x)}|\Det(I+t\grad^2S^{in}(y_j(t,x))|\,,
$$
and therefore
$$
\ba
\sqrt{(-i)^N\Det(I+t\grad^2S^{in}(y_j(t,x))}&
\\
=|\Det(I+t\grad^2S^{in}(y_j(t,x))|^{1/2}\sqrt{-i}^{N-\cM_j(t,x)}\sqrt{i}^{\cM_j(t,x)}&
\\
=|\Det(I+t\grad^2S^{in}(y_j(t,x))|^{1/2}e^{-i\pi(N-\cM_j(t,x)-\cM_j(t,x))/4}&\,.
\ea
$$
Thus
\be\lb{MaslovFact}
\ba
\frac{a^{in}(y_j(t,x))e^{i\phi(t,x,y_j(t,x))/\eps-iN\pi/4}}{\sqrt{(-i)^N\Det(I+t\grad^2S^{in}(y_j(t,x))}}&
\\
=
\frac{a^{in}(y_j(t,x))e^{i\phi(t,x,y_j(t,x))/\eps-iN\pi/4}e^{+i\pi(N-2\cM_j(t,x))/4}}{|\Det(I+t\grad^2S^{in}(y_j(t,x))|^{1/2}}&
\\
=
\frac{a^{in}(y_j(t,x))e^{i\phi(t,x,y_j(t,x))/\eps}i^{-\cM_j(t,x)}}{|\Det(I+t\grad^2S^{in}(y_j(t,x))|^{1/2}}&\,.
\ea
\ee

The proof of (\ref{WKBLim}) follows from substituting (\ref{MaslovFact}) in (\ref{StatPh0}), and using (\ref{NonStatPh0}) and (\ref{StatPh0}) in the sum (\ref{Sumj0}) representing $\psi_\eps$.
\end{proof}

\smallskip
Several remarks are in order.

\smallskip
First, even if $S^{in}$ does not satisfy the assumption (\ref{Sublin}), for $(t,x)\in\bR\times\bR^N\setminus C$, the set $F_t^{-1}(\{x\})$ of critical points of the map $y\mapsto\phi(t,x,y)$ is discrete by the implicit function theorem. Its 
intersection with $\Supp(a^{in})$ is therefore finite whenever $a^{in}$ has compact support.

\smallskip
Another observation bears on the physical meaning of the points $y_j(t,x)$. These points satisfy
$$
y_j(t,x)+t\grad S^{in}(y_j(t,x))=x.
$$
In other words, $x$ is the position at time $t$ of the free particle leaving the position $y_j(t,x)$ at time $t=0$ with velocity $\grad S^{in}(y_j(t,x))$, according to Newton's first law of classical mechanics. 

\smallskip
Next we recall the following well known geometric interpretation of the caustic $C$. In the case $N=1$, the system (\ref{Caustic}) reduces to
$$
\left\{
\ba
{}&y+t(S^{in})'(y)=x\,,
\\
&1+t(S^{in})''(y)=0\,.
\ea
\right.
$$
The first equality is the defining equation for a family of straight lines in $\bR_t\times\bR_x$ parametrized by $y\in\bR$ --- the trajectories of the particle in classical mechanics --- while the second equality follows from deriving the 
first with respect to the parameter $y$. By eliminating the parameter $y$ between both equations, we see that the caustic $C$ is in general the envelope of the family of straight lines defined by the first equation. 

\smallskip
The integer $\cM_j(t,x)$ can be interpreted in terms of the notion of Maslov index, as explained in the last section of this paper. Notice that the integer $\cM_j(t,x)\to 0$ as $t\to 0$, so that, for each  $x\in\bR^N$, one has $\cM_j(t,x)=0$ 
for all $t$ near $0$. Furthermore, $\cM_j$ is a locally constant function of $(t,x)\in\bR\times\bR^N\setminus C$, and therefore the presence of the Maslov index in (\ref{WKBLim}) is equivalent to a phase shift of an integer multiple of 
$\pi/2$ whenever $(t,x)$ is moved across $C$ from one connected component of $\bR\times\bR^N$ to the next.


\section{The case $V\not=0$}


The explicit representation formula (\ref{Sol}) obviously is a major ingredient in the proof of Theorem \ref{T-WKBV=0}. In the case $V\not=0$, there is no explicit formula analogous to (\ref{Sol}) giving the solution of the Cauchy problem 
for the Schr\"odinger equation in general.

However, under assumption (\ref{Cinftyb}), there exists a FIO that is a parametrix for the operator 
$$
G_\eps(t):=e^{i\tfrac{t}\eps\left(\tfrac12\eps^2\Dlt_x-V\right)}\,.
$$
Perhaps the simplest approach to this most important result is Theorem 2.1 in \cite{LapSig}, whose main features are recalled below.

Consider the action
\begin{equation}\label{act}
S(t,x,\xi):=\int_0^t\left(\tfrac12|\Xi_s(x,\xi)|^2-V(X_s(x,\xi))\right)ds
\end{equation}
Given $T>0$, we shall have to deal with the class of phase functions 
$$
\vphi\equiv\vphi(t,x,y,\eta)\in C\hbox{ of class }C^\infty\hbox{ on }[0,T)\times\bR^N\times\bR^N\times\bR^N)
$$
satisfying the conditions
\be\lb{Condphi}
\left\{
\ba
{}&\vphi(t,X_t(y,\eta),y,\eta)=S(t,y,\eta)\,,
\\
&D_x\vphi(t,X_t(y,\eta),y,\eta)=\Xi_t(y,\eta)\,,
\\
&iD^2_x\vphi(t,x,y,\eta)\le 0\hbox{ is independent of }x\,,
\\
&\Det(D_{x\eta}\vphi(t,X_t(y,\eta),y,\eta))\not=0\hbox{ for each }(t,y,\eta)\in[0,T)\times\bR^N\times\bR^N\,.
\ea
\right.
\ee

Pick $\chi\in C^\infty_c(\bR^N\times\bR^N)$ and $T>0$. Then, for any phase function $\vphi$ satisfying (\ref{Condphi}) and any $n\ge 0$, there exists $A_n\equiv A_n(t,y,\eta,\eps)\in C^\infty_c([0,T]\times\bR^N\times\bR^N)[\eps]$ such 
that the FIO $G_n(t)$ with Schwartz kernel
\be\lb{FIOPara}
G_{\eps,n}(t,x,y)=\int A_n(t,y,\eta,\eps)e^{i\vphi(t,x,y,\eta)/\eps}\frac{d\eta}{(2\pi\eps)^N}
\ee
satisfies
\be\lb{ErrPara}
\sup_{0\le t\le T}\|(G_\eps(t)-G_{\eps,n}(t))\chi(x,-i\eps\d_x)\|_{\cL(L^2(\bR^N))}\le C[V,T,\chi]\eps^{n-2N}\,.
\ee
In this inequality the notation $\chi(x,-i\eps\d_x)$ designates the pseudo-differential operator defined by the formula
$$
\chi(x,-i\eps\d_x)\phi(x):=\iint_{\bR^N\times\bR^N}e^{i(x-y)\cdot\eta/\eps}\chi(x,\eta)\phi(y)\frac{dyd\eta}{(2\pi\eps)^N}\,.
$$

Taking Theorem 2.1 in \cite{LapSig} for granted, one arrives at the following description of the classical limit of (\ref{ClassSchrod}). It is stated without proof in Appendix 11 of \cite{ArnoldMech} or as Theorem 5.1 in \cite{ArnoldMaslov}.

Let $S^{in}\in C^2(\bR^N)$ satisfy (\ref{Sublin}) and let $C$ be defined as in (\ref{Caustic}); let $\cN(t,x)$ and $y_j(t,x)$ be defined as in Proposition \ref{P-Smooth} for each $(t,x)\in\bR\times\bR^N\setminus C$. Let $J_t(y)$ be defined 
as in (\ref{DefJ}).

\begin{Thm}\lb{T-WKBV}
Let $a^{in}\in C^m_c(\bR^N)$ and $S^{in}\in C^{m+1}(\bR^N)$ satisfy (\ref{Sublin}), with regularity index $m>6N+4$. For all $\eps>0$ and all $(t,x)\in\bR_+\times\bR^N\setminus C$, set
\be\lb{WKBSol}
\Psi_\eps(t,x)\!=\!\sum_{j=1}^{\cN(t,x)}\!\frac{a^{in}(y_j(t,x))}{J_t(y_j(t,x))^{1/2}}e^{iS_j(t,x)/\eps}i^{-\cM_j(t,x)}\,,
\ee
with
$$
S_j(t,x):=S^{in}(y_j(t,x))+S(t,y_j(t,x),\grad S^{in}(y_j(t,x)))\,,\quad j=1,\ldots,\cN(t,x)\,,
$$
where $S(t,y,\xi)$ is given by \eqref{act} and $\cM_j(t,x)\in\bZ$ for all $(t,x)\in\bR_+\times\bR^N\setminus C$ is constant on each connected component of $\bR\times\bR^N$ where $j\le\cN$.

Then the solution $\psi_\eps$ of the Cauchy problem (\ref{ClassSchrod}) satisfies
\be\lb{WKBLimV}
\psi_\eps(t,x)=\Psi_\eps(t,x)+R_\eps^1(t,x)+R_\eps^2(t,x)
\ee
for all $T>0$, where
$$
\sup_{0\le t\le T}\|R_\eps^1\|_{L^2(B(0,R))}=O(\eps)\hbox{ for all }R>0
$$
and 
$$
\sup_{(t,x)\in K}|R_\eps^2(t,x)|=O(\eps)\hbox{ for each compact }K\subset\bR_+\times\bR^N\setminus C
$$
as $\eps\to 0^+$.
\end{Thm}

\smallskip
A fairly classical computation shows that each one of the phases $S_j$ is a solution of the Hamilton-Jacobi ``eikonal'' equation
$$
\d_tS_j(t,x)+\tfrac12|\grad_xS_j(t,x)|^2+V(x)=0
$$
for $(t,x)$ belonging to the union of all the connected components of $\bR\times\bR^N\setminus C$ where $j\le\cN(t,x)$.

\smallskip
The integer $\cM_j(t,x)$ in (\ref{WKBSol}) is  defined precisely in the proof (see formula (\ref{Step4}) below). It turns out to be identical to the Maslov index of the path
$$
[0,t]\ni s\mapsto\Phi_s(y_j(t,x),\grad S^{in}(y_j(t,x)))\,,
$$
as explained in the last section of this paper.

\smallskip
A last remark bears on the formulation of the classical limit of the Schr\"odinger equation in terms of the Wigner transform. Given $t>0$, assume that $a^{in}$ has compact
support with
$$
F_t(\Supp(a^{in})\cap C_t=\varnothing\,,
$$
and that the vectors $\grad_xS_j(t,x)$ are pairwise different for $x\in F_t(\Supp(a^{in})$. Denote by $W_\eps[\psi_\eps]$ the Wigner transform of $\psi_\eps$, defined by the formula
$$
W_\eps[\psi_\eps](t,x,\xi)=\tfrac1{(2\pi)^N}\int_{\bR^N}e^{-i\xi\cdot y}\psi_\eps(t,x+\tfrac12\eps y)\overline{\psi_\eps(t,x-\tfrac12\eps y)}dy\,.
$$
Using Theorem \ref{T-WKBV} above and Proposition 1.5 in \cite{GMMP} shows that
$$
W_\eps[\psi_\eps](x,\cdot)\to\sum_{j=1}^{\cN(t,x)}\frac{|a^{in}(y_j(t,x))|^2}{J_t(y_j(t,x))}\de_{\grad_xS_j(t,x)}
$$
in $\cS'(\bR^N_x\times\bR^N_\xi)$ as $\eps\to 0$. The same result can be obtained by a completely different method, avoiding the use of Theorem \ref{T-WKBV}, and for much less regular initial phase functions $S^{in}$ (typically,
for $S^{in}$ having second order derivatives in the Lorentz space $L^{N,1}_{loc}$): see \cite{BGMP} for more details on the classical limit of the Schr\"odinger equation with rough phase functions.

\smallskip
Before giving the proof of Theorem \ref{T-WKBV} in detail, it is worth having in mind the differences and similarities between this proof and that of Theorem \ref{T-WKBV=0}. The construction of the parametrix in \cite{LapSig}
replaces the \textit{exact} solution of the Schr\"odinger equation by and \textit{approximate} solution, whose structure is similar to the right hand side of (\ref{Sol}), at the expense of an error estimate controlled by using the
bound in (\ref{ErrPara}). We conclude by applying the stationary phase method to this approximate solution exactly as in the proof of Theorem \ref{T-WKBV=0}.

Notice that the reference \cite{LapSig} contains a proof of Theorem \ref{T-WKBV} in the special case where $S^{in}$ is a \textit{linear} function. We shall mostly follows the argument in \cite{LapSig}, except that the case of a
general phase $S^{in}$, not necessarily linear, requires an additional trick (see Lemma \ref{L-detABCD} below).

\begin{proof}
Let $\chi(x,\xi)=\chi_1(x)\chi_2(\xi)$ with $\chi_1,\chi_2\in C^\infty_c(\bR^N)$, satisfying
$$
\indc_{B(0,R)}(x)\le\chi_1(x)\le\indc_{B(0,R+1)}\hbox{ and }\indc_{B(0,Q)}(\xi)\le\chi_2(\xi)\le\indc_{B(0,Q+1)}
$$
for all $x,\xi\in\bR^N$, where $R>0$ and $Q$ is to be chosen later. Pick $n>2N$; then
$$
\ba
\psi_\eps(t,\cdot)-G_{\eps,n}(t)\chi(x,-i\eps\d_x)\psi_\eps^{in}
&=
G_\eps(t)(1-\chi(x,-i\eps\d_x))\psi_\eps^{in}
\\
&+(G_\eps(t)-G_{\eps,n}(t))\chi(x,-i\eps\d_x)\psi_\eps^{in}\,.
\ea
$$
Since $G_\eps(t)$ is a unitary group on $L^2(\bR^N)$
$$
\ba
\|\psi_\eps(t,\cdot)-G_{\eps,n}(t)\chi(x,-i\eps\d_x)\psi_\eps^{in}\|_{L^2(\bR^N)}
\le
\|(1-\chi(x,-i\eps\d_x))\psi_\eps^{in}\|_{L^2(\bR^N)}
\\
+
\|(G_\eps(t)-G_{\eps,n}(t))\chi(x,-i\eps\d_x)\psi_\eps^{in}\|_{L^2(\bR^N)}
\\
\le
\|(1-\chi(x,-i\eps\d_x))\psi_\eps^{in}\|_{L^2(\bR^N)}
\\
+C_{T,Q}\eps^{n-2N}\|a^{in}\|_{L^2(\bR^N)}
\ea
$$
for all $t\in[0,T]$, where $C_{Q,T}=C[V,T,\chi]$.

Now, $\chi(x,-i\eps\d_x)\psi_\eps^{in}=\chi_1(x)\chi_2((-i\eps\d_x)\psi_\eps^{in}$ and since $\Supp(a^{in})\subset B(0,R)$
$$
\ba
\|(1-\chi(x,-i\eps\d_x))\psi_\eps^{in}\|_{L^2(\bR^N)}
=
\|\chi_1(1-\chi_2(-i\eps\d_x))\psi_\eps^{in}\|_{L^2(\bR^N)}
\\
\le
\|(1-\chi_2(-i\eps\d_x))\psi_\eps^{in}\|_{L^2(\bR^N)}
\\
=
(2\pi)^{-N}\|(1-\chi_2(\eps\xi))\widehat{\psi_\eps^{in}}\|_{L^2(\bR^N)}
\\
\le
(2\pi)^{-N}\|\indc_{[Q/\eps,\infty)}(|\xi|)\widehat{\psi_\eps^{in}}\|_{L^2(\bR^N)}
\ea
$$
Since
$$
\widehat{\psi_\eps^{in}}(\zeta/\eps)=\int_{\bR^N}e^{-i(\zeta\cdot x-S^{in}(x))/\eps}a^{in}(x)dx
$$
we conclude from estimate (7.7.1') in \cite{Hor1} that
$$
|\widehat{\psi_\eps^{in}}(\zeta/\eps)|\le\frac{C\|a^{in}\|_{W^{m,\infty}(\bR^N)}}{(|\zeta|-\|\grad S^{in}\|_{L^\infty(B(0,R)})^m}\eps^m
$$
provided that $\Supp(a^{in})\subset B(0,R)$ and $|\zeta|>1+\|\grad S^{in}\|_{L^\infty(B(0,R)}$. Therefore
\be\lb{Step1}
\ba
\|\psi_\eps(t,\cdot)-G_{\eps,n}(t)\chi(x,-i\eps\d_x)\psi_\eps^{in}\|_{L^2(\bR^N)}
\le
C_{T,Q}\|a^{in}\|_{L^2(\bR^N)}\eps^{n-2N}
\\
+
C\|(1+|\zeta|)^{-m}\|_{L^2(\bR^N)}\|a^{in}\|_{W^{m,\infty}(\bR^N)}\eps^m
\ea
\ee
for all $m>N/2$.

Next we analyze the term
\be\lb{OscilInt}
\ba
G_{\eps,n}(t)&\chi(x,-i\eps\d_x)\psi_\eps^{in}(x)
\\
&=
\iint\!\!\!\iint A_n(t,y,\eta,\eps)a^{in}(z)\chi_2(\zeta)e^{i(\vphi(t,x,y,\eta)+\zeta\cdot(y-z)+S^{in}(z))/\eps}\frac{dzd\zeta dyd\eta}{(2\pi\eps)^{2N}}
\ea
\ee
with the stationary phase method.

Choose $\vphi$ of the form\footnote{For each $v\in\bR^N$, the tensor $v\otimes v$ is identified with the matrix with entries $v_iv_j$, where $v_i$ is the $i$th component of the vector $v$ in the canonical basis of $\bR^N$. For $A,B\in M_N(\bR)$,
the notation $A:B$ designates $\Tr(A^TB)$.}
$$
\vphi(t,x,y,\eta)=S(t,y,\eta)+(x-X_t(y,\eta))\cdot\Xi_t(y,\eta)+iB:(x-X_t(y,\eta))^{\otimes 2}
$$
where the matrix $B=B^T>0$ is constant (see formula (2.7) \cite{LapSig} and the following Remark 2.1). Critical points of the phase in the oscillating integral (\ref{OscilInt}) are defined by the system of equations\footnote{If $f\in C^1(\bR^N;\bR^N)$, and if $f_i(x)$ designates the $i$th component of $f(x)$ in the canonical basis of $\bR^N$, the notation $Df(x)$ designates the Jacobian matrix of $f$ at the point $x$, i.e. the matrix whose entry at the $i$th row and the $j$th column is the partial
derivative $\d_{x_j}f_i(x)$.}
$$
\left\{
\ba
{}&-\zeta+DS^{in}(z)=0\,,
\\
&y-z=0\,,
\\
&\d_yS(t,x,y)-D_yX_t(y,\eta)^T\cdot\Xi_t(y,\eta)+D_y\Xi_t(y,\eta)^T\cdot(x-X_t(y,\eta))
\\
&-iB:(x-X_t(y,\eta))\otimes D_yX_t(y,\eta)+\zeta=0\,,
\\
&\d_\eta S(t,x,y)-D_\eta X_t(y,\eta)^T\cdot\Xi_t(y,\eta)+D_\eta\Xi_t(y,\eta)^T\cdot(x-X_t(y,\eta))
\\
&-iB:(x-X_t(y,\eta))\otimes D_\eta X_t(y,\eta)=0\,.
\ea
\right.
$$
At this point, we recall formulas (3.1-2) from \cite{LapSig}
\be\lb{SySeta}
\ba
{}&\d_yS(t,y,\eta)=D_yX_t(y,\eta)^T\cdot\Xi_t(y,\eta)-\eta\,,
\\
&\d_\eta S(t,y,\eta)=D_\eta X_t(y,\eta)^T\cdot\Xi_t(y,\eta)\,,
\ea
\ee
together with the following definitions 
$$
\ba
Y(t,y,\eta):=D_y\Xi_t(y,\eta)-iBD_y X_t(y,\eta)\,,
\\
Z(t,y,\eta):=D_\eta\Xi_t(y,\eta)-iBD_\eta X_t(y,\eta)\,.
\ea
$$
Thus the critical points of the phase in (\ref{OscilInt}) are given by
$$
\left\{
\ba
{}&\zeta=DS^{in}(z)\,,
\\
&y=z\,,
\\
&(x-X_t(y,\eta))^TY(t,y,\eta)+\zeta=\eta\,,
\\
&(x-X_t(y,\eta))^TZ(t,y,\eta)=0\,.
\ea
\right.
$$
Since the matrix $Z$ is invertible by Lemma 4.1 of \cite{LapSig}, we conclude that the system of equations above is equivalent to
$$
\left\{
\ba
{}&\zeta=DS^{in}(z)\,,
\\
&y=z\,,
\\
&\zeta=\eta\,,
\\
&x=X_t(y,\eta)\,.
\ea
\right.
$$
In other words, 
$$
F_t(z)=x\,,\quad y=z\,,\quad\zeta=\eta=DS^{in}(z)\,.
$$
Assuming that $(t,x)\notin C$, we apply Proposition \ref{P-Smooth} and conclude that the set of critical points of the phase in (\ref{OscilInt}) is of the form
$$
\left\{
\ba
{}&y=z=y_j(t,x)\,,
\\
&\zeta=\eta=DS^{in}(y_j(t,x))\,,
\ea
\right.
\qquad\qquad j=1,\ldots,\cN(t,x)\,.
$$
At this point, we apply the stationary phase method (Theorem 7.7.5 in \cite{Hor1}). First we need to compute the Hessian of the phase in (\ref{OscilInt}) 
at its critical points. One finds
$$
H_j(t,x):=\left(\begin{matrix}
D^2S^{in}\,\,&\,\,-I\,\,&\,\,0\,\,&0\,\,\\
-I&0&+I&0\\
0&+I&-Y^TD_yX_t&-Y^TD_\eta X_t-I\\
0&0&-Z^TD_yX_t&-Z^TD_\eta X_t
\end{matrix}\right)_{y=y_j(t,x)\atop \eta=DS^{in}(y_j(t,x))}
$$
and it remains to compute $\Det(H_j(t,x))$. Adding the first row of $H_j(t,x)$ to the third row, one finds that
$$
\ba
\Det H_j(t,x)&=
\left|\begin{matrix}
D^2S^{in}\,\,&\,\,-I\,\,&\,\,0\,\,&0\,\,\\
-I&0&+I&0\\
D^2S^{in}&0&-Y^TD_yX_t&-Y^TD_\eta X_t-I\\
0&0&-Z^TD_yX_t&-Z^TD_\eta X_t
\end{matrix}\right|_{y=y_j(t,x)\atop \eta=DS^{in}(y_j(t,x))}
\\
&=\left|\begin{matrix}
-I&+I&0\\
D^2S^{in}&-Y^TD_yX_t&-Y^TD_\eta X_t-I\\
0&-Z^TD_yX_t&-Z^TD_\eta X_t
\end{matrix}\right|_{y=y_j(t,x)\atop \eta=DS^{in}(y_j(t,x))}
\\
&=\left|\begin{matrix}
-I&0&0\\
D^2S^{in}&-Y^TD_yX_t+D^2S^{in}&-Y^TD_\eta X_t-I\\
0&-Z^TD_yX_t&-Z^TD_\eta X_t
\end{matrix}\right|_{y=y_j(t,x)\atop \eta=DS^{in}(y_j(t,x))}
\ea
$$
where the last equality follows from adding the first column in the right hand side of the second equality to the second column. Eventually, one finds that
$$
\ba
\Det H_j(t,x)
=(-1)^N\left|\begin{matrix}
-Y^TD_yX_t+D^2S^{in}&-Y^TD_\eta X_t-I\\
-Z^TD_yX_t&-Z^TD_\eta X_t
\end{matrix}\right|_{y=y_j(t,x)\atop \eta=DS^{in}(y_j(t,x))}
\ea
$$
which is computed as follows. First
$$
\left|\begin{matrix}
I&-(Z^{-1}Y)^T\\
0&I
\end{matrix}\right|
\left|\begin{matrix}
-Y^TD_yX_t+D^2S^{in}&-Y^TD_\eta X_t-I\\
-Z^TD_yX_t&-Z^TD_\eta X_t
\end{matrix}\right|
=
\left|\begin{matrix}
D^2S^{in}&-I\\
-Z^TD_yX_t&-Z^TD_\eta X_t
\end{matrix}\right|
$$
so that
$$
\Det H_j(t,x)
=
(-1)^N\left|\begin{matrix}
D^2S^{in}&-I\\
-Z^TD_yX_t&-Z^TD_\eta X_t
\end{matrix}\right|_{y=y_j(t,x)\atop \eta=DS^{in}(y_j(t,x))}
$$
On the other hand
\be\lb{SchurForm}
\ba
\left|\begin{matrix}
D^2S^{in}&-I\\
-Z^TD_yX_t&-Z^TD_\eta X_t
\end{matrix}\right|
&=
(-1)^N\Det(Z^TD_yX_t+Z^TD_\eta X_tD^2S^{in})
\\
&=(-1)^N\Det(Z)\Det(D_yX_t+D_\eta X_tD^2S^{in})
\ea
\ee
by the following elementary lemma (that is a variant of the Schur complement formula in a special case: see for instance Proposition 3.9 on pp. 40-41 in \cite{SerreMat}).

\begin{Lem}\lb{L-detABCD}
Let $A,B,C,D\in M_N(\bC)$. If $AB=BA$, one has
$$
\left|\begin{matrix}
A&B\\
C&D\end{matrix}\right|
=\Det(DA-CB)\,.
$$
\end{Lem}

\smallskip
Therefore
\be\lb{FlaHess}
\ba
\Det H_j(t,x)&=\Det(Z)\Det(D_yX_t+D_\eta X_tD^2S^{in})\Big|_{y=y_j(t,x),\,\eta=DS^{in}(y_j(t,x))}
\\
&=\Det(Z(y_j(t,x),DS^{in}(y_j(t,x)))\Det(DF_t(y_j(t,x)))
\ea
\ee
where $F_t$ is defined in (\ref{DefF}).

Pick a nonempty closed ball $\bbb\subset\bR\times\bR^N\setminus C$, let $\cN_\bbb=\cN(t,x)$ for all $(t,x)\in\bbb$, and let
$$
K_j=\{(y_j(t,x),\grad S^{in}(y_j(t,x)))\,|\,(t,x)\in\bbb\}\,,\qquad j=1,\ldots,\cN_\bbb\,.
$$
Assuming that $\bbb$ is of small enough radius, $K_j\cap K_k=\varnothing$ for $j\not=k\in\{1,\ldots,\cN_\bbb\}$. Let $\ka_j\in C^\infty_c(\bR^{2N})$ for all $j=1,\ldots,\cN_\bbb$, such that
$$
\left\{
\ba
\ka_j\ge 0\hbox{ and }\ka_j\rstr_{K_j}=1\,,\quad j=1,\ldots,\cN_\bbb\,,
\\
\hbox{while }\ka_j\ka_k=0\hbox{ for }j\not=k\in\{1,\ldots,\cN_\bbb\}\,,
\ea
\right.
$$
Applying Theorem 7.7.1 in \cite{Hor1} shows that
\be\lb{Step2}
\ba
\sup_{(t,x)\in \bbb}\left|\iint\!\!\!\iint A_n(t,y,\eta,\eps)a^{in}(z)\chi_2(\zeta)e^{i(\vphi(t,x,y,\eta)+\zeta\cdot(y-z)+S^{in}(z))/\eps}
	\frac{dzd\zeta dyd\eta}{(2\pi\eps)^{2N}}\right.
\\
-\sum_{j=1}^{\cN_\bbb}\iint\!\!\!\iint A_n(t,y,\eta,\eps)a^{in}(z)\chi_2(\zeta)\ka_j(y,\eta)\ka_j(z,\zeta)
\\
\left. e^{i(\vphi(t,x,y,\eta)+\zeta\cdot(y-z)+S^{in}(z))/\eps}\frac{dzd\zeta dyd\eta}{(2\pi\eps)^{2N}}\right|=O(\eps)
\ea
\ee
as $\eps\to 0$. 

Next we set
$$
\ba
I_j(t,x,\eps):=\iint\!\!\!\iint A_n(t,y,\eta,\eps)a^{in}(z)\chi_2(\zeta)\ka_j(y,\eta)\ka_j(z,\zeta)
\\
\times e^{i(\vphi(t,x,y,\eta)+\zeta\cdot(y-z)+S^{in}(z))/\eps}\frac{dzd\zeta dyd\eta}{(2\pi\eps)^{2N}}
\ea
$$
for $j=1,\ldots,\cN_\bbb$. By Theorem 7.7.5 in \cite{Hor1}, we conclude that
\be\lb{Step3}
\ba
\sup_{(t,x)\in\bbb}|I_j(t,x,\eps)-A_0(t,y_j(t,x),\grad S^{in}(y_j(t,x)),0)a^{in}(y_j(t,x))\chi_2(\grad S^{in}(y_j(t,x)))
\\
\times e^{i(\vphi(t,x,y_j(t,x),\grad S^{in}(y_j(t,x)))+S^{in}(y_j(t,x)))/\eps}(\Det H_j(t,x))^{-1/2}|=O(\eps)
\ea
\ee
as $\eps\to 0$. Our choice of $\chi_2$ and $\vphi$ implies that $\chi_2(\grad S^{in}(y_j(t,x)))=1$ and
$$
\vphi(t,x,y_j(t,x),\grad S^{in}(y_j(t,x)))=S(t,y_j(t,x),\grad S^{in}(y_j(t,x)))
$$
so that
$$
(\vphi(t,x,y_j(t,x),\grad S^{in}(y_j(t,x)))+S^{in}(y_j(t,x)))=S_j(t,x)\,.
$$
By formula (2.13) in \cite{LapSig}
\begin{equation}\label{root}
A_0(t,y_j(t,x),\grad S^{in}(y_j(t,x)),0)=\sqrt[cont]{\Det(Z(t,y_j(t,x),DS^{in}(y_j(t,x)))}\,,
\end{equation}
where the notation $\sqrt[cont]{z}$ designates  the analytic continuation of the square-root along the path $s\mapsto\Det(Z(s,y_j(t,x),DS^{in}(y_j(t,x))$. This analytic continuation is uniqueley defined since $\Det(Z(s,y_j(t,x),DS^{in}(y_j(t,x)))\not=0$ 
for all $s$: see for instance section 1.3 in chapter 8 of \cite{Ahlfors}. According to \eqref{root}, (\ref{FlaHess}), Lemma 5.1 and formula (5.15) in \cite{LapSig}, we have
\be\lb{Step4}
\ba
A_0(t,y_j(t,x),\grad S^{in}(y_j(t,x)),0)(\Det H_j(t,x))^{-1/2}
\\
=|\Det(DF_t(y_j(t,x)))|^{-1/2}e^{i\pi\nu_j(t,x)/2}=J(y_j(t,x))^{-1/2}e^{-i\pi\cM_j(t,x)/2}\,,
\ea
\ee
where $\cM_j(t,x)$ is an integer.

Putting together (\ref{Step1})-(\ref{Step2})-(\ref{Step3})-(\ref{Step4}) concludes the proof of Theorem \ref{T-WKBV}.
\end{proof}

\begin{proof}[Proof of Lemma \ref{L-detABCD}]
If $A$ is nonsingular and $AB=BA$, one has
$$
\ba
\left|\begin{matrix}
A&B\\
C&D\end{matrix}\right|
=
(-1)^N\left|\begin{matrix}
A&B\\
C&D\end{matrix}\right|
\left|\begin{matrix}
A^{-1}&B\\
0&-A\end{matrix}\right|
&=
(-1)^N\left|\begin{matrix}
I&0\\
CA^{-1}&CB-DA\end{matrix}\right|
\\
&=(-1)^N\Det(CB-DA)=\Det(DA-CB)\,.
\ea
$$
Since both sides of the identity above are continuous functions of $A$ and the set of nonsingular matrices $GL_N(\bC)$ is dense in $M_N(\bC)$, this identity holds for all $A\in M_n(\bC)$ such that $AB=BA$.
\end{proof}

\smallskip
In the case of a linear phase function $S^{in}$ treated in section 5 of \cite{LapSig}, one has $\grad^2S^{in}\equiv 0$. Therefore, the determinant in (\ref{SchurForm}) reduces to 
$$
\left|\begin{matrix}
0&-I\\
-Z^TD_yX_t&-Z^TD_\eta X_t
\end{matrix}\right|
$$
which can be explicitly computed without difficulty, since this determinant is blockwise triangular. In particular, the simpler situation considered in  section 5 of \cite{LapSig} does not require using Lemma \ref{L-detABCD}.


\section{Maslov-Index mit menschlichem Antlitz}


The purpose of this section is to explain, in the simplest possible manner, how the integers $\cM_j(t,x)$ appearing in formulas (\ref{WKBLim}) and (\ref{WKBSol}) are related to the Maslov index.

\subsection{Generalities on the Maslov index} 


There are several notions of Maslov index in the literature. The original definition can be found in Maslov's treatise \cite{Masl} or \cite{ArnoldMaslov}. Closely related indices have been subsequently defined by Leray (in \S 2 of\cite{Leray}) 
and H\"ormander (in chapter XXI of \cite{Hor3}) --- see also chapters I.7 and IV.3 in \cite{GuillStern}, and especially \cite{Souriau} for a lucid presentation of all these notions of Maslov index and how they are related. The present section 
recalls some material from Arnold's short and precise presentation \cite{ArnoldMaslov} of the subject.

The phase space $\bR_x^N\times\bR_\xi^N$ is endowed with the standard symplectic $2$-form $\si:=d\xi_1\wedge dx_1+\ldots+d\xi_N\wedge dx_N$. A linear subspace $\l$ of $\bR_x^N\times\bR_\xi^N$ is called Lagrangian if and 
only if $\Dim\l=N$ and $\si(u,v)=0$ for all $u,v\in\l$. An example of Lagrangian subspace of $\bR_x^N\times\bR_\xi^N$ is $T^*_0:=\{0\}\times\bR^N$. The Lagrangian Grassmanian $\L(N)$ is the set of Lagrangian subspaces of 
$\bR_x^N\times\bR_\xi^N$. For each $k=0,\ldots,N$ define $\L^k(N):=\{\l\in\L(N)\,|\,\Dim(\l\cap T^*_0)=k\}$. A linear subspace $\l$ of $\bR_x^N\times\bR_\xi^N$ belongs to $\L^0(N)$ iff $\l$ is defined by an equation of the form $\xi=Ax$ 
with $A=A^T\in M_N(\bR)$. The Lagrangian Grassmanian $\L(N)$ is a $C^\infty$ manifold of dimension $\tfrac12N(N+1)$ (Corollary 3.1.4 in \cite{ArnoldMaslov}), and $\L^k(N)$ is a submanifold of codimension $\tfrac12k(k+1)$ in 
$\L(N)$ (Lemma 3.2.1 in \cite{ArnoldMaslov}). An important subset of the Lagrangian Grassmanian is the \textit{Maslov cycle} $\cM:=\L(N)\setminus\L^0(N)=\L^1(N)\cup\ldots\cup\L^N(N)$; it has codimension $1$ in $\L(N)$ (see section 
3.2, especially Corollary 3.2.2 in \cite{ArnoldMaslov} for a proof that the homological boundary of $\cM$ is $0$). 

Let us define an orientation on $\cM$. For each $\th\in\bR$ and each $(x,\xi)\in\bR_x^N\times\bR_\xi^N$, define $R[\th](x,\xi)=(x\cos\th+\xi\sin\th,\xi\cos\th-x\sin\th)$. Since $R[\th]$ is the Hamiltonian flow of $H(x,\xi):=\tfrac12(|x|^2+|\xi|^2)$,
it defines a symplectomorphism of the phase space $\bR_x^N\times\bR_\xi^N$. In particular $R[\th]\l\in\L(N)$ for each $\l\in\L(N)$ and each $\th\in\bR$. The Maslov cycle is oriented by the following prescription: let $\l\in\L^1(N)$; then the 
path $\th\mapsto R[\th]\l$ crosses $\cM$ exactly once for $\th$ near $0$, at $\l$ for $\th=0$, and does so from the negative side of $\cM$ to the positive side of $\cM$, as $\th$ increases in $(-\eta,\eta)$ for $\eta>0$ small enough (section 3.5, 
especially Lemmas 3.5.1,3.5.2 and 3.5.3 in \cite{ArnoldMaslov}).

Let now $[0,t]\ni s\mapsto\l(s)\in\L(N)$ be a $C^1$ path such that $\l(0),\l(t)\in\L^0(N)$ and such that $\{\l(s)\,|\,0<s<t\}\cap\cM\subset\L^1(N)$ with transverse intersection. The Maslov index $\mu(\l)$ of the path $\l$ is its intersection index 
with the Maslov cycle $\cM$ oriented as above. In other words
$$
\mu(\l)=\sum_{\l(s)\in\cM}\Sign(s)\,,
$$
where $\Sign(s)=+1$ if $\l(s+t)$ crosses $\cM$ from the negative to the positive side of $\cM$ as $t$ increases near $0$, and $\Sign(s)=-1$ if $\l(s+t)$ crosses $\cM$ from the positive to the negative side of $\cM$ as $t$ increases near 
$0$ (see section 2.2 and Definition 3.6.1 in \cite{ArnoldMaslov}).

There exists an alternate definition of the Maslov index for closed paths in $\L(N)$. Identifying $\bR_x^N\times\bR_\xi^N$ with $\bC^N=\bR_\xi^N+i\bR_x^N$, we recall that the unitary group $U(N)$ acts transitively on $\L(N)$ (Lemma 
1.2 in \cite{ArnoldMaslov}). Thus, for each $\l\in\L(N)$, there exists $u\in U(N)$ such that $\l=uT^*_0$; besides, if $u,u'\in U(N)$ and $uT^*_0=u'T^*_0$, then $uu^T=u'(u')^T$. In other words, one can identify $\l$ with $uu^T$ where $u$ 
is any element of $U(N)$ such that $uT^*_0=\l$. This defines a map $\Det^2:\,\L(N)\to\bS^1$ by $\Det^2(\l):=\Det(u)^2$ where $uT^*_0=\l$. Let now $\bS^1\ni s\mapsto\l(s)\in\L(N)$ be a closed continuous path; the Maslov index $\mu(\l)$ 
of $\l$ is defined as the winding number of the composed map $\Det^2\circ\l:\,\bS^1\to\bS^1$, i.e.
$$
\mu(\l)=\hbox{degree}(\Det^2\circ\l)\,.
$$
(See section 1.5 in \cite{ArnoldMaslov}). If $\l$ is a $C^1$ closed path on $\L(N)$ intersecting $\cM$ transversally on $\L^1(N)$, both definitions of the Maslov index coincide (Theorem 1.5 in \cite{ArnoldMaslov}).

\subsection{The Maslov index and Hamiltonians of the form $\tfrac12|\xi|^2+V(x)$.}


The orientation of the Maslov cycle $\cM$ is obviously crucial in the definition of the Maslov index recalled above. For that reason, computing the Maslov index of a path is in general a rather complicated task. However, when the path
is defined by the linearized Hamiltonian flow of a Hamiltonian that is convex in the momentum variable, this computation is considerably simplified. Indeed, in that case, the orientation of the Maslov cycle plays no role as such a path
always crosses the Maslov cycle in the same direction. Therefore, computing the Maslov index of such a path reduces to counting how many times it intersects the Maslov cycle. In other words, the Maslov index reduces to the more
classical notion of Morse index (see for instance \cite{Milnor} \S 15, especially Theorem 15.1) in this case. This observation can be found in \cite{Masl}, first without  proof on p. 151, and as the result of a rather lengthy argument on 
p. 297. See also Theorem 5.2, given without proof in \cite{ArnoldMaslov}. The lemma below gives a short proof of this fact in the special case of a Hamiltonian of the form $\tfrac12|\xi|^2+V(x)$, which is all that we need in the context 
of the classical limit of the Schr\"odinger equation. 

Let $W\in C_b(\bR;M_N(\bR))$ such that $W(t)=W(t)^T$ for each $t\in\bR$, and let $t\mapsto S(t,t_0)\in M_{2N}(\bR)$ be the solution of the Cauchy problem
$$
\frac{d}{dt}S(t,t_0)=\left(\begin{matrix} 0&I\\ W(t)&0\end{matrix}\right)S(t,t_0)\,,\qquad S(t_0,t_0)=I\,.
$$
(The linearized Hamiltonian system defined by the Hamiltonian $\tfrac12|\xi|^2+V(x)$ and the symplectic form $d\xi_1\wedge dx_1+\ldots+d\xi_N\wedge dx_N$ is exactly of this form, with $W=-\grad^2V$.) The matrix $S(t,t_0)$ is 
symplectic because 
$$
\left(\begin{matrix} 0&-I\\ I&0\end{matrix}\right)\left(\begin{matrix} 0&W(t)\\ I&0\end{matrix}\right)+\left(\begin{matrix} 0&I\\ W(t)&0\end{matrix}\right)\left(\begin{matrix} 0&-I\\ I&0\end{matrix}\right)=0\,.
$$
In particular $S(t,t_0)\l\in\L(N)$ whenever $\l\in\L(N)$. 

\begin{Lem}\lb{L-Fond}
Let $\l_0\in\L^0(N)$ and set $\l(t)=S(t,t_0)\l_0$. If $\l(t_1)=S(t_1,t_0)\l_0\in\L^1(N)$, the path $t\mapsto\l(t)$ intersects the Maslov cycle $\cM$ transversally at $\l(t_1)$ from the negative side to the positive side of the Maslov cycle as 
$t$ increases near $t_1$.
\end{Lem}

\begin{proof}
Choose a system of orthonormal coordinates $p_1,\ldots,p_N$ in $\bR^N_\xi$ such that the line $\l(t_1)\cap T_0^*$ is transverse to the hyperplane of equation $p_1=0$ in $T_0^*$. Choose orthonormal coordinates $q_1,\ldots,q_N$ 
in $\bR^N_x$ that are conjugate to $p_1,\ldots,p_N$, i.e. such that the symplectic form $\si=dp_1\wedge dq_1+\ldots+dp_N\wedge dq_N$. The change of coordinates takes the form
$$
\left(\begin{matrix} R&0\\ 0&R\end{matrix}\right)
$$
where $R\in O_N(\bR)$. In these new coordinates, the differential equation defining $S(t,t_0)$ keeps the same form, up to replacing $W(t)$ with $RW(t)R^T$. For simplicity, we keep the same notation for $S(t,t_0)$ and $W(t)$ in these
new variables.

Define
$$
\hat p_1=q_1\,,\quad\hat q_1=-p_1\,,\quad \hat p_j=p_j\hbox{ and }\hat q_j=q_j\quad\hbox{ for }j=2,\ldots,N\,.
$$
In these coordinates, any Lagrangian space $\l'$ such that $T_0^*\cap\l'$ is a line transverse to the hyperplane of equation $p_1=0$ in $T_0^*$ is defined by an equation of the form $\hat p=L'\hat q$ with $L'=(L')^T\in M_N(\bR)$ such 
that $L'_{11}=0$. 

Denote $I_1:=\Diag(1,0,\ldots,0)\in M_N(\bR)$ and $I':=I-I_1$, and set
$$
J:=\left(\begin{matrix} I'&I_1\\ -I_1&I'\end{matrix}\right)
$$
so that $J(q,p)^T=(\hat q,\hat p)^T$. Straightforward computations show that
$$
JS(t_1+\tau,t_1)J^T=\left(\begin{matrix}I\,-\,\tau I_1W(t_1)I'&\tau I'\!+\!\tau I_1W(t_1)I_1\\	\\ -\tau I'W(t_1)I'\!-\!\tau I_1&I+\tau I'W(t_1)I_1\end{matrix}\right)+O(\tau)^2\,.
$$

Let $L=L^T\in M_N(\bR)$ be such that the Lagrangian subspace $\l(t_1)$ has equation $\hat p=L\hat q$. For $|\tau|\ll 1$, the Lagrangian subspace $\l(t_1+\tau)$ has equation $\hat p= L(\tau)\hat q$, where
$$
\ba
L(\tau)=&(-\tau I'W(t_1)I'\!-\!\tau I_1+(I+\tau I'W(t_1)I_1)L)
\\
&\times(I\,-\,\tau I_1W(t_1)I'+(\tau I'\!-\!\tau I_1W(t_1)I_1)L)^{-1}+O(\tau^2)
\ea
$$
so that, in the limit as $\tau\to 0$, one has
$$
\ba
\frac1{\tau}(L(\tau)-L)\to &-I'W(t_1)I'-I_1+I'W(t_1)I_1L\\&+LI_1W(t_1)I'-LI'L-LI_1W(t_1)I_1L\,.
\ea
$$
Assume that $\l(t_1)\in\L^1(N)$ and that $T^*_0\cap\l(t_1)$ is transverse to the hyperplane of equation $\hat p_1=0$ on $T^*_0$, so that $L_{11}=0$. Let $e_1$ be the first vector in the canonical basis of $\bR^N$; then, in the limit as 
$\tau\to 0$, one has
$$
\frac1\tau(e_1|(L(\tau)-L)e_1)=\frac1\tau(e_1|L(\tau)e_1)\to-1-|I'Le_1|^2\le -1\,.
$$

In the special case $W=-I$, one has $S(t,t_0)=R[t-t_0]$. Therefore, near $\l(t_1)$, the computation above shows that the positive side of $\L^1(N)$ consists of Lagrangian subspaces $\l'$ of equation $\hat p=L'\hat q$ with $L'_{11}<0$.

Besides, the computation above also shows that, for all $W$, the path $t\mapsto S(t,t_1)\l(t_1)$ crosses $\L^1(N)$ in the same direction as for $W=-I$, i.e. from the negative side to the positive side as $t$ increases near $t_1$.
\end{proof}

\subsection{The free case}


Let $A=A^T\in GL_N(\bR)$, and let $\l(0)$ be the Lagrangian subspace of equation $\xi=Ax$ in $\bR_x^N\times\bR_\xi^N$. Let $\Phi_t(x,\xi)=(x+t\xi,\xi)$ be the free flow defined on $\bR_x^N\times\bR_\xi^N$ for each $t\in\bR$, which is
the Hamiltonian flow of $H(x,\xi)=\tfrac12|\xi|^2$. Set $\l(s):=\Phi_s(\l(0))\in\L(N)$ for each $s\in\bR$. 

First assume that $A$ has $N$ distinct eigenvalues $\a_1>\ldots>\a_N$, and let $t>0$ be such that $I+tA$ is invertible. Consider the path $\l_t:\,[0,t]\ni s\mapsto\l(s)\in\L(N)$. Obviously $\l_t(s):=\{(x+sAx,Ax)\,|\,x\in\bR^N\}$. Hence 
$\l_t(s)\in\L^0(N)$ if $I+sA$ is invertible, and has equation $\xi=A(I+sA)^{-1}x$. 

If $I+sA$ is not invertible, then $\Ker(I+sA)$ has dimension $1$  since $A$ has simple eigenvalues, and therefore $\l_t(s)\cap T^*_0=\{0\}\times\Ker(I+sA)$ has dimension $1$. Thus the path $\l_t$ can only intersect the Maslov cycle $\cM$ 
on its regular part $\L^1(N)$. By Lemma \ref{L-Fond}, it always does so in the same direction, from the negative to the positive side of $\cM.$ Hence the Maslov index of the path $\l_t$ is
$$
\mu(\l_t)=\#\{\a_j\,|\,0<-1/\a_j<t\}=\#\{\a_j\,|\,1+t\a_j<0\}\,.
$$

\smallskip
Next we treat the general case, where $A$ may have multiple eigenvalues, still denoted $\a_1\ge\ldots\ge\a_N$ and counted with their multiplicities, and compute the Maslov index $\mu(\l_t)$ of the path $\l_t$ defined above.

Pick $A'=(A')^T$ near $A$ with distinct eigenvalues and such that $A'A=AA'$. Set $A(\tau)=(1-\tau)A+\tau A'$. Assume that $I+tA$ is invertible; by choosing $A'$ sufficiently close to $A$, one can assume that $I+tA(\tau)$ is invertible 
for each $\tau\in[0,1]$. For each $B=B^T\in\ M_N(\bR)$, denote by $\l[B]$ the Lagrangian subspace of equation $\xi=Bx$. Consider now the family indexed by $\tau\in[0,1]$ of closed paths $\g_\tau$ defined as follows
$$
\ba
\g_\tau(s)&=\Phi_{ts}\l[A]\,,&\quad\hbox{ for }0\le s\le 1\,,
\\
\g_\tau(s)&=\Phi_t\l[A(\tau(s-1))]\,,&\quad\hbox{ for }1<s<2\,,
\\
\g_\tau(s)&=\Phi_{t(3-s)}\l[A(\tau)]\,,&\quad\hbox{ for }2\le s\le 3\,,
\\
\g_\tau(s)&=\l[A(\tau(4-s))]\,,&\quad\hbox{ for }3<s<4\,.
\ea
$$
Letting $\tau\to 0$ shows that $\g_1$ is homotopic to $\l_t-\l_t$ (i.e. the path $\l_t$ followed by its opposite). By the homotopy invariance of the degree, using the definition of Maslov's index for closed continuous paths shows that 
$\mu(\g_\tau)=0$. Since $I+tA(\tau)$ is invertible for $\tau\in[0,1]$, the arcs of $\g_1$ corresponding to $s\in[1,2]$ and $s\in[3,4]$ never cross $\cM$. Therefore 
$$
\mu(\l_t)=-\mu(\g_1\rstr_{[2,3]})\,,
$$
from which we conclude that 
$$
\mu(\l_t)=\#\{\a_j\,|\,1+t\a_j<0\}\,.
$$

\subsection{The non free case}


Denote by $\Phi_t=(X_t,\Xi_t)$ the Hamiltonian flow of 
$$
H(x,\xi):=\tfrac12|\xi|^2+V(x)\,,
$$
let $(t,x)\notin C$, and let $\cL_0:=\{(y,\grad S^{in}(y))\hbox{ s.t. }y\in\bR^N\}$. Denote $\cL_s:=\Phi_s(\cL_0)$ for each $s\in\bR$. Set $\g(s):=\Phi_s(y_j(t,x),\grad S^{in}(y_j(t,x)))$ for each $s\in\bR$ (where $y_j$ has been defined in
Proposition \ref{P-Smooth}, and $\l_0:=T_{\g(0)}\cL_0$. Let $\l_s:=D\Phi_s(\g(0))\l_0$ for each $s\in\bR$; observe that $\g(s+s')=\Phi_{s'}(\g(s))$ and that $D\Phi_{s'}(\g(s))\l(s):=\l(s+s')$ by the chain rule. With this notation, the point
$X_s(y_j(t,x),\grad S^{in}(y_j(t,x)))$ belongs to $C_s$ if and only if $\l(s)\in\overline\L^1(N)$.

\smallskip
Assume that $\l(s)\in\L^1(N)$ whenever $\g(s)\in C_s$ for $0<s<t$. 

\smallskip
In that case, by Lemma \ref{L-Fond}, the path $s\mapsto\l(s)$ always crosses the Maslov cycle transversally from the negative to the positive side as $\g(s)\in C_s$. Therefore, the Maslov index of the path $[0,t]\ni s\mapsto\l(s)\in\L(N)$ 
is in this case
$$
\mu((\l(s))_{0\le s\le t})=\#\{s\in[0,t]\hbox{ s.t. }X_s(y_j(t,x),\grad S^{in}(y_j(t,x)))\in C_s\}\,.
$$
The set $\{s\in[0,t]\hbox{ s.t. }X_s(y_j(t,x),\grad S^{in}(y_j(t,x)))\in C_s\}$ is therefore finite and henceforth denoted by $0<s_1<s_2<\ldots<s_n<t$. For $k=1,\ldots,n$, there exists $\phi_k\not=0$ such that $\Ker(DF_{s_k}(y_j(t,x)))=\bR\phi_k$.

\smallskip
Consider on the other hand the matrix
$$
M(s):=\left(\begin{matrix}
-Y_{s}^TD_yX_s+D^2S^{in}&-Y_{s}^TD_\eta X_s-I\\
-Z_{s}^TD_yX_s&-Z_{s}^TD_\eta X_s
\end{matrix}\right)_{y=y_j(t,x)\atop \eta=DS^{in}(y_j(t,x))}
$$
where
$$
Y_{s}=D_y\Xi_s-iBD_yX_s\,,\quad\hbox{ and }Z_{s}=D_\eta\Xi_s-iBD_\eta X_s\,.
$$
Denote 
$$
M_1(s):=\left(\begin{matrix}
-D_y\Xi_s^TD_yX_s+D^2S^{in}&-D_y\Xi_s^TD_\eta X_s-I\\
-D_\eta\Xi_s^TD_yX_s&-D_\eta\Xi_s^TD_\eta X_s
\end{matrix}\right)_{y=y_j(t,x)\atop \eta=DS^{in}(y_j(t,x))}
$$
and
$$
M_2(s):=\left(\begin{matrix}D_yX_s^T\\ D_\eta X_s^T\end{matrix}\right)B(D_yX_s\,\,\,D_\eta X_s)\rstr_{y=y_j(t,x)\atop \eta=DS^{in}(y_j(t,x))}
$$
so that
$$
M_1(s)=M_1(s)^T=\Re(M(s))\,,\quad M_2(s)=M_2(s)^T=\Im(M(s,\th))\ge 0\,.
$$
Since $\Det(M(s))=(-1)^N\Det(Z_{s})\Det(DF_s(y_j(t,x))$ and $Z_{s}$ is invertible for each $s\in[0,t]$, one has $\Det(M(s))=0$ at $s=s_1<s_2<\ldots<s_n$ only in $[0,t]$. Besides, 
$$
\ba
M(s_k)(\phi\oplus\psi)=0&\Rightarrow\psi=D^2S^{in}(y_j(t,x))\phi\hbox{ and }DF_{s_k}(y_j(t,x))\phi=0
\\
&\Rightarrow M_1(s_k)(\phi\oplus\psi)=M_2(s_k)(\phi\oplus\psi)=0\,,
\ea
$$
the converse being obvious. Hence
$$
\Ker(M(s_k))=\bC(\phi_k\oplus D^2S^{in}(y_j(t,x))\phi_k)\,,
$$
while
$$
\Ker(M_1(s_k))\cap\Ker(M_2(s_k))=\bR(\phi_k\oplus D^2S^{in}(y_j(t,x))\phi_k)\,.
$$
Denote $V_k:=\bR(\phi_k\oplus D^2S^{in}(y_j(t,x))\phi_k)$; since $M_j(s_k)$ is a real symmetric matrix and $M_j(s_k)(V_k)\subset V_k$ for $j=1,2$, one has $M_j(s_k)(V_k^\perp)\subset V_k^\perp$. Consider now the linear space
$W_k:=V_k^\perp\oplus iV_k^\perp$; one has $M(s_k)W_k\subset W_k$ and
$$
\bC^{2N}=\Ker(M(s_k,\th))\oplus W_k\,.
$$
In particular $M(s_k)\rstr_{W_k}$ is invertible on $W_k$. Therefore, computing the characteristic polynomial of $M(s_k)$ in this decomposition of $\bC^{2N}$, we find that
$$
\Det(\l I_{2N}-M(s_k))=\l\Det(\l I_{2N-1}-M(s_k)\rstr_{W_k})\,,
$$
so that $\l=0$ is a simple root of the characteristic polynomial of $M(s_k)$.

By the implicit function theorem, there exists a local $C^1$ function $s\mapsto\l_k(s)$ defined near $s_k$ and such that 
$$
\l_k(s_k)=0\,,\quad\hbox{ and }\Det(\l_k(s)I_{2N}-M(s))=0\hbox{ for all }s\hbox{ near }s_k\,.
$$
Moreover, there exist a $C^1$ vector field $s\mapsto\psi_k(s)\not=0$ defined near $s_k$ such that 
$$
M(s)\psi_k(s)=\l_k(s)\psi_k(s)\,,\quad\psi_k(s_k)=\phi_k\oplus D^2S^{in}(y_j(t,x))\phi_k\,.
$$
Differentiating in $s$ at $s=s_k$, one finds
$$
\dot M(s_k)\psi_k(s_k)+M(s_k)\dot\psi_k(s_k)=\dot\l_k(s_k)\psi_k(s_k)
$$
and observing that $M(s)=M(s)^T$, one concludes that
$$
\psi_k(s_k)^T\dot M(s_k)\psi_k=\dot\l_k(s_k)\psi_k(s_k)^T\psi_k(s_k)\,,
$$
so that 
$$
\dot\l_k(s_k)=\frac{\psi_k(s_k)^T\dot M(s_k)\psi_k(s_k)}{|\psi_k(s_k)|^2}\,.
$$
Notice that $\psi_k(s_k)^T\psi_k(s_k)=|\psi_k(s_k)|^2>0$ since $\psi_k(s_k)=\phi_k\oplus D^2S^{in}(y_j(t,x))\phi_k$ belongs to $\bR^{2N}\setminus\{0\}$.

Observe that
$$
\d_sD_yX_s=D_y\Xi_s\,,\qquad\hbox{ and }\d_sD_\eta X_s=D_\eta\Xi_s\,,
$$
so that
$$
\ba
\dot M(s_k)=-\left(\begin{matrix}\d_sY_{s_k}^T\\ \d_sZ_{s_k}^T\end{matrix}\right)(D_yX_{s_k}\,\,\,D_\eta X_{s_k})\rstr_{y=y_j(t,x)\atop \eta=DS^{in}(y_j(t,x))}
\\
-\left(\begin{matrix}Y_{s_k}^T\\ Z_{s_k}^T\end{matrix}\right)(D_y\Xi_{s_k}\,\,\,D_\eta\Xi_{s_k})\rstr_{y=y_j(t,x)\atop \eta=DS^{in}(y_j(t,x))}
\ea
$$
and
$$
\dot M(s_k)\psi_k(s_k)=-\left(\begin{matrix}Y_{s_k}^T\\ Z_{s_k}^T\end{matrix}\right)(D_y\Xi_{s_k}\,\,\,D_\eta\Xi_{s_k})\rstr_{y=y_j(t,x)\atop \eta=DS^{in}(y_j(t,x))}\psi_k(s_k)\,.
$$
Thus
$$
\ba
\psi_k(s_k)^T\dot M(s_k)\psi_k(s_k)=&i\psi_k(s_k)^T\left(\begin{matrix}D_yX_{s_k}^T\\ D_\eta X_{s_k}^T\end{matrix}\right)B(D_y\Xi_{s_k}\,\,\,D_\eta\Xi_{s_k})\rstr_{y=y_j(t,x)\atop \eta=DS^{in}(y_j(t,x))}\psi_k(s_k)
\\
&-\psi_k(s_k)^T\left(\begin{matrix}D_y\Xi_{s_k}^T\\ D_\eta\Xi_{s_k}^T\end{matrix}\right)(D_y\Xi_{s_k}\,\,\,D_\eta\Xi_{s_k})\rstr_{y=y_j(t,x)\atop \eta=DS^{in}(y_j(t,x))}\psi_k(s_k)\,.
\ea
$$
Since $\psi_k(s_k)=\phi_k\oplus D^2S^{in}(y_j(t,x))\phi_k$ with $DF_{s_k}(y_j(t,x))\phi_k=0$, the first term on the right hand side is
$$
\ba
i\psi_k(s_k)^T\left(\begin{matrix}D_yX_{s_k}^T\\ D_\eta X_{s_k}^T\end{matrix}\right)B(D_y\Xi_{s_k}\,\,\,D_\eta\Xi_{s_k})\rstr_{y=y_j(t,x)\atop \eta=DS^{in}(y_j(t,x))}\psi_k(s_k)
\\
=
i\phi_k^TDF_{s_k}(y_j(t,x))^TB(D_y\Xi_{s_k}\,\,\,D_\eta\Xi_{s_k})\rstr_{y=y_j(t,x)\atop \eta=DS^{in}(y_j(t,x))}\psi_k(s_k)=0\,.
\ea
$$
Hence
$$
\psi_k(s_k)^T\dot M(s_k)\psi_k(s_k)=-|(D_y\Xi_{s_k}\,\,\,D_\eta\Xi_{s_k})\psi_k(s_k)|^2\le 0\,.
$$
This last inequality can obviously not be an equality since 
$$
(D_yX_{s_k}\,\,\,D_\eta X_{s_k})\rstr_{y=y_j(t,x)\atop \eta=DS^{in}(y_j(t,x))}\psi_k(s_k)=DF_{s_k}(y_j(t,x))\phi_k=0\,,
$$
while
$$
\Det\left(\begin{matrix}D_yX_s&D_\eta X_s\\ D_y\Xi_s &D_\eta\Xi_s\end{matrix}\right)=1
$$
and $\psi_k(s_k)\not=0$. Therefore
$$
\dot\l_k(s_k)<0\,.
$$

Now, the function $[0,t]\ni s\mapsto\sqrt{\Det(M(s)/i)}\in\bC$ is continuous, and its argument has jump discontinuities for $s=s_k$ for $k=1,\ldots,n$ only. Each time $s\in[0,t]$ crosses one of the values $s_k$, the jump in the argument 
of $\sqrt{\Det(M(s)/i)}$ is exactly the jump in the argument of $\sqrt{\l_k(s)/i}$, and the previous computation shows that this jump is exactly $+\tfrac{\pi}2$. Hence 
$$
\cM_j(t,x)=n=\mu((\l(s))_{0\le s\le t})\,.
$$



\end{document}